\RequirePackage{fix-cm}
\documentclass[smallcondensed]{svjour3}     
\smartqed  
\usepackage{graphicx}
\usepackage{amssymb,color}

\usepackage{amsthm,amsmath}

\newcommand{\N}{\mathbb N}

\newcommand{\R}{\mathbb R}

\newcommand\normsup{\norm_{\infty}}
\newcommand\norm{\parallel \cdot \parallel}

\newcommand\qiC{\mbox{qi}C}
\newcommand\qidual{\mbox{qi}C^*}

\newcommand\cero{0_X}

\begin{document}

\title{On dentability and cones with a large dual\thanks{The author M. A. Melguizo Padial has been supported by project MTM2017-86182-P (AEI/FEDER, UE). The author Fernando Garc\'i{}a-Casta\~no has been partially supported by MINECO and FEDER (MTM2014-54182), by Fundaci\'{o}n S\'{e}neca - Regi\'{o}n de Murcia (19275/PI/14), and by MTM2017-86182-P (AEI/FEDER, UE).}
}


\author{Fernando Garc\'i{}a-Casta\~no   \and  M. A. Melguizo Padial
}


\institute{Fernando Garc\'i{}a-Casta\~no \at
Departamento de Matem\'atica Aplicada. Escuela Polit\'ecnica
Superior. Universidad de Alicante.  03080 San Vicente del Raspeig.
Alicante, Spain\\
              Tel.: +34 965903400 x Ext. 2823 \\
              Orcid: 0000-0002-8352-8235\\
              \email{fernando.gc@ua.es}           
           \and
           M. A. Melguizo Padial \at
             Departamento de Matem\'atica Aplicada. Escuela Polit\'ecnica
Superior. Universidad de Alicante.  03080 San Vicente del Raspeig.
Alicante, Spain
}

\date{Received: date / Accepted: date}

\maketitle

\begin{abstract}
In this paper we provide some equivalences on dentability in normed spaces. Among others we prove: the origin is a denting point of a pointed cone $C$  if and only if it is a point of continuity for such a cone and $\overline{C^*-C^*}=X^*$; $x$ is a denting point of a convex set $A$ if and only if $x$ is a point of continuity and a weakly strongly extreme point of $A$.  We also analize how our results help us to shed some light on several open problems in the literature.

MSC 46B20
\end{abstract}
\section{Introduction}
Throughout the paper $X$ will denote a normed space, $\norm$ the norm of $X$, $X^*$ the dual space of $X$, $\cero$ the origin of $X$, and $\R_+$ the set of nonnegative real numbers. We call \emph{weak} the weak topology on $X$ and \emph{weak}$^*$ the weak star topology on its dual $X^*$. A nonempty convex subset $C$ of $X$ is called a \emph{cone} if $\alpha C \subset C$, $\forall \alpha \in \R_+$. Fixed a cone $C$, we define the \emph{dual cone} for $C$ by $C^{*}:=\{f \in X^* \colon f(c)\geq 0\, , \forall c \in C\}$ --which is a weak$^*$-closed subset of $X^*$-- and the \emph{bidual cone} for $C$ by $C^{**}:=\{T \in X^{**}\colon T(f)\geq 0,\, \forall f \in C^*\}$ --which is a weak$^*$-closed subset of $X^{**}$--. A cone $C \subset X$ is said to be \emph{quasi-generating} (resp. \emph{generating}) if $\overline{C-C}=X$ (resp. if $C-C=X$). It is clear that generating cones are quasi-generating, but the converse is not true (for example, a proper dense subspace). 
The notions of generating and quasi-generating cones have been widely studied and appear in many results regarding ordered vector spaces, see for example \cite{ali-tour,Borwein1992}. Fixed a normed space $X$, by q$\mathcal{G}^*$ we will denote the family of cones in $X$ that have a quasi-generating dual cone.  

Let us introduce the notion of denting point. It connects with the Radon--Nikod\'ym property, with renorming theory, and with optimization theory. Before introducing such a notion we need to fix some notation and terminology. A set $H$ is called an \emph{open half space} of $X$ if it is a weakly open set of the form $H=\{x \in X \colon f(x)<\lambda\}$ for some $f \in X^*\setminus\{0_{X^*}\}$ and $\lambda \in \R$. We denote $H$ briefly by $\{f<\lambda\}$. Let $A\subset X$, a \emph{slice} of $A$ is a nonempty intersection of $A$ with an open half space of $X$. We denote by conv($A$) (resp. by $\overline{\mbox{conv}}$($A$)) the \emph{convex hull} of $A$ (resp. the \emph{closed convex hull} of $A$). Besides, we denote by $B_{\epsilon}(x)$ the open ball with centre $x \in X$ and radius $\epsilon>0$ and by $\overline{B_{\epsilon}(x)}$ the corresponding closed ball. In case $x=\cero$ and $\epsilon=1$ the former set is denoted by $B_X$ and the latter by $\overline{B_X}$. Let $A$ be a convex subset of $X$, $x \in A$ is said to be a \emph{denting point} of $A$ if $x \not \in \overline{\mbox{conv}}(A\setminus B_{\epsilon}(x))$, $\forall \epsilon>0$. By the Hahn--Banach theorem, $x$ is a denting point of $A$ if and only if for every $\epsilon>0$ there exists a slice of $A$ containing $x$ with diameter less than $\epsilon$.  Another notion involved in this work is that of point of continuity. Let $A$ be a convex subset of $X$, $x \in A$ is said to be a \emph{point of continuity} for $A$ if the identity map $(A,\mbox{weak})\rightarrow (A,\norm)$ is continuous at $x$.

The notion of point of continuity is, in general, weaker than the notion of denting point. However, it is still an open problem to clarify if both notions coincide for closed cones in normed spaces. Such a problem raised in \cite[Conclusions]{Gong1995}, paper in which these properties were applied in vector optimization theory (more applications of these properties in vector optimization can be found in \cite{Bednarczuk1998,Casini2010b,Petschke1990}). In \cite{GARCIACASTANO20151178,Kountzakis2006} were given some results characterizing the notion of denting point for a cone in a normed space in terms of the notion of point of continuity plus an extra assumption on the cone. Let us note that such extra assumptions seem not to be related to the closedness of the cone. However, in the context of Banach spaces both notions (dentability and point of continuity) become equivalent for closed cones. The former equivalence was shown in  \cite{Daniilidis2000} as a direct consequence of a well known characterization of denting points of a closed convex bounded set in a Banach space, namely \cite[Theorem]{Lin1988}.  In this paper we continue the research we began in \cite{GARCIACASTANO20151178} about the link between dentability and point of continuity in normed spaces (not necessarily complete) and for arbitrary cones (not necessarily closed). In this line of research, we provide now some new results which have a wide range of applicability --as we will see below--, and we solve --among others-- Problem 2.6 we stated in \cite{GARCIACASTANO20151178}.

In the following paragraphs we will present a brief overview of the main results of this paper together with their connections with other results and problems in the literature. Our first contribution is Theorem \ref{Teorema_caracte_denting}, in which we stablish --among others-- the equivalence between dentability and point of continuity for cones in q$\mathcal{G}^*$. Roughly speaking, the former equivalence can be interpreted as meaning that the notions of denting point and point of continuity in a cone coincide when the dual cone is ``large enough''. This result generalizes \cite[Theorem 4]{Kountzakis2006} and, in addition, it provides the following: a negative answer to Gong's question in  \cite[Conclusions]{Gong1995} for cones in q$\mathcal{G}^*$ (see Corollary~\ref{Coro_Gong}); the equivalence between the following results on density of Arrow, Barankin and Blackwell's type, \cite[Theorem~3.2~(a)]{Gong1995} and \cite[Corollary 4.2]{Petschke1990}, for cones in q$\mathcal{G}^*$ (see Corollary~\ref{Coro_density_theorems}); and a positive answer to \cite[Problem 5]{Kountzakis2006} for cones in q$\mathcal{G}^*$ (see Corollary~\ref{Coro_Kountz}).  Luckily, q$\mathcal{G}^*$ is a large class of cones. It contains, for example, the normal cones introduced by Krein in \cite{Krein1940}. Normal cones are useful in the theory of ordered topological vector spaces (see \cite{ali-tour}) and in the study of extremal problems in differential and integral equations (see \cite{Krasnoselskii1964}). 


Returning to \cite{GARCIACASTANO20151178} --paper of which this work is a continuation--, it is worth pointing out that \cite[Proposition 2.5]{GARCIACASTANO20151178} is a cornerstone in the proof of the main theorem in  \cite{GARCIACASTANO20151178}. Such a proposition characterizes the property that the origin is a weakly strongly extreme point of a pointed closed cone. However, it was unknown if the above mentioned characterization remains true if we drop down to the cone the assumption of closedness. Namely, such a question was stated in \cite[Problem~2.6]{GARCIACASTANO20151178}. In this paper we solve the former question in the negative (see Remark~\ref{Remark_Problem_Prop_Vicente_Mejorada}). Let us note that such a solution is a consequence of Theorem~\ref{Tma_Prop_Vicente_para_Cono_mejorada} in this paper, which provides a characterization of the property that the origin is a locally weakly strongly extreme point of a pointed (not necessarily closed) cone.




As has been said before, the equivalence of dentability and point of continuity for closed cones in Banach spaces was provided by Daniilidis who made use of the characterization \cite[Theorem]{Lin1988}. Let us point out that such a characterization states that a point  of a closed, convex, and bounded set in a Banach space is a denting point if and only if it is a point of continuity and an extreme point for such a set. However, the former equivalence does not work for noncomplete normed spaces (see \cite[Example]{Lin1989}) and, to our knowledge, a version for normed spaces  is required. In this paper, by means of Theorem \ref{Teorema_caracte_denting_point}, we provide such a version changing the notion of extreme point by that of weakly strongly extreme point and assuming --only-- the convexity of the involved set. 

The paper is organized as follows. In Section \ref{Sec_Quasi_gen_cones} we state and prove Theorems \ref{Teorema_caracte_denting}, \ref{Tma_Prop_Vicente_para_Cono_mejorada}, and \ref{Teorema_caracte_denting_point} mentioned above. In Section \ref{Secc_Normal_cones} we apply the equivalence between dentability and point of continuity for cones in q$\mathcal{G}^*$ from Theorem \ref{Teorema_caracte_denting} to solve --for cones in q$\mathcal{G}^*$-- some open problems in the literature we mentioned before (Gong's question and so on). Finally, in Section \ref{Secc_Examples} we  again apply the former equivalence from Theorem \ref{Teorema_caracte_denting}, but this time to several noncomplete ordered normed spaces. As a result, we obtain some geometric properties of their corresponding canonical order cones.

%

\section{Results on dentability}\label{Sec_Quasi_gen_cones}
The main objective in this section is to state and prove Theorems \ref{Teorema_caracte_denting}, \ref{Tma_Prop_Vicente_para_Cono_mejorada}, and \ref{Teorema_caracte_denting_point} just commented in the introduction. However, before stating Theorem \ref{Teorema_caracte_denting}, we have to say that it is a complement to Theorem 1.1 in \cite{GARCIACASTANO20151178} in the sense that we provide here more caracterizations to the property of dentability of a cone at the origin. In fact, thanks to the use of some equivalences from \cite[Theorem 1.1]{GARCIACASTANO20151178}, we will be able to provide a short proof of Theorem \ref{Teorema_caracte_denting}. Next, we state such equivalences. 
\begin{theorem}(From \cite[Theorem 1.1]{GARCIACASTANO20151178})\label{Tma_JMAA}
Let $X$ be a normed space and $C \subset X$ a pointed cone. The following are equivalent:
\item[(i)] $0_X$ is a denting point of $C$.
\item[(ii)] $C$ has a bounded slice.
\item[(iii)] The interior of $C^*$ in $X^*$ is not empty.
\end{theorem}
Next, we provide the first of our main results. Before, we need to fix some notation. Let us fix a convex set $A$ and $x\in A$. We say that $x$ is an \emph{extreme point} of $A$ if $x$ does not belong to any non degenerate line segment in $A$. On the other hand, let us remind that a cone $C$ is said to be \emph{pointed} if $C\cap (-C)=\{\cero\}$ (equivalently, if $\cero$ is an extreme point of $C$). It is clear that if $\cero$ is a denting point of a cone $C$, then $C$ is pointed. Let $C\subset X$ be a cone, \emph{the order on $X$ given by $C$} is defined as $x\leq y$ $\Leftrightarrow y-x\in C$, $\forall x$, $y \in X$. In this context it is defined the \emph{order interval} $[x,y]$ as the set $\{z\in X \colon x\leq z \leq y\}$ and it is said that $c\in C$ is an \emph{order unit} of $C$ if $X=\bigcup_{n\geq 1}[-nc,nc]$. Given any cone in a normed space, we always consider the corresponding order on the normed space given by such a cone. 

\begin{theorem}\label{Teorema_caracte_denting}
Let $X$ be a normed space and $C\subset X$ a pointed cone. The following are equivalent:
\item[(i)] $0_X$ is a denting point of $C$.
\item[(ii)] There exist $n \in \N$, $\{f_i\}_{i=1}^n\subset C^{*}$, and $\{\lambda_i\}_{i=1}^n\subset (0,+\infty)$ such that the set $\bigcap_{i=1}^n\{f_i<\lambda_i\}\cap C$ is bounded.
\item[(iii)] $\cero$ is a point of continuity for $C$ and $C\in \mathrm{q}\mathcal{G}^*$.
\item[(iv)] $C^*$ has an order unit.
\item[(v)] There exists $\{f_n\}_{n\geq 1}\subset C^*$ such that $X^*=\bigcup_{n\geq 1}[-nf_n,nf_n]$. 
\end{theorem}

The main contributions of the former result with respect to \cite[Theorem~1.1]{GARCIACASTANO20151178} are, on the one hand, the use of quasi-generating dual cones as a qualification condition for the equivalence between point of continuity and dentability which has a wider range of applicability --as we will see later--. On the other hand, we state explicitly the relation of quasi-generating dual cones with order units. Let us remind that $c \in  C$ is said to be a \emph{quasi-interior point of $X$} if $\overline{\bigcup_{n \geq 1}[-nc,nc]}=X$. The set of all quasi-interior points is denoted by $\qiC$.  It is clear that $\qidual \not = \emptyset$ implies $C\in \mathrm{q}\mathcal{G}^*$. Thus, the equivalence between assertions (i) and (iii) above generalizes \cite[Theorem 4]{Kountzakis2006}. After the examples in Section \ref{Secc_Examples} it is clear that we can not remove the hypothesis of point of continuity from (iii). Finally, let us note that assertions (ii), (iv), and (v) suggest that for $\cero$ to be a denting point of $C$ it is only necessary to have enough positive elements in $X^*$, i. e., that $C$ has a ``large enough'' dual cone.

The following lemmata will be required in the proof of Theorem \ref{Teorema_caracte_denting}. The first two lemmas are known, however we prove the second one for the completeness of the text.
\begin{lemma}\label{lema_interval_1}
Let $X$ be a normed space, $x \in X$, and $C \subset X$ a cone. Then $[-x,x]$ is a convex and symmetric set. Besides, $[-x,x]\not = \emptyset$ if and only if $\cero \in [-x,x]$ if and only if $x \in C$.
\end{lemma}
Here and subsequently, Int$A$ stands for the interior of the set $A \subset X$.
\begin{lemma}\label{lema_interval_2}
Let $X$ be a normed space, $x \in X$, and $C \subset X$ a cone. The following are equivalent:
\begin{itemize}
\item[(i)] Int$[-x,x]\not = \emptyset$.
\item[(ii)] $\cero \in $ Int$[-x,x]$.
\item[(iii)] $x \in$ Int$C$.
\end{itemize}
\end{lemma}
\begin{proof}
(i)$\Rightarrow$(ii). If $z \in$ Int$[-x,x]$, then $-z \in $ Int$[-x,x]$, so $\cero \in $ Int$[-x,x]$. (ii)$\Rightarrow$(iii). If $\exists \epsilon>0$ such that $\epsilon B_X\subset [-x,x]=(-x+C)\cap (x-C)$, then clearly $B_{\epsilon}(x)\subset C$. (iii)$\Rightarrow$(ii). Assume $\exists \epsilon>0$ such that $B_{\epsilon}(x)\subset C$. Then $\epsilon B_X\subset -x+C$. Since $B_X$ is symmetric, $\epsilon B_X\subset x-C$. Thus, $\cero \in $ Int$[-x,x]$. 
\end{proof}
\begin{lemma}\label{lema_order_unit}
Let $X$ be a Banach space and $C \subset X$ a closed cone. The following statements hold true:
\begin{itemize}
\item[(i)] $x$ is an order unit of $C$ if and only if $x \in$Int$C$.
\item[(ii)] There exists $(c_n)_n \subset C$ such that $\bigcup_{n \geq 1}[-nc_n,nc_n]=X$ if and only if Int$C$ is not empty.
\end{itemize}
\end{lemma}
\begin{proof}
(i) $\Rightarrow$. Assume that $x$ is an order unit. Then, by the Baire category theorem, there exists $n \in \N$ such that $[-nx,nx]$ has a nonempty interior. Now, by Lemma~\ref{lema_interval_2}, $x \in$ Int$C$. Let us prove $\Leftarrow$. If $x \in$ Int$C$, then, again by Lemma \ref{lema_interval_2}, $[-x,x]$ has a nonempty interior. Since $[-x,x]$ is a convex and symmetric subset of $X$, we get that $x$ is an order unit. 

(ii) $\Rightarrow$. Again by the Baire category theorem, there exists $n$ such that $[-c_n,c_n]$ has a nonempty interior. Now, Lemma \ref{lema_interval_2} applies and $c_n \in$ Int$C$. To finish, $\Leftarrow$ is a consequence of (i).
\end{proof}

Let us denote by $\overline{C}^{\mbox{weak}^*}$ the closure of $C$ in $X^{**}$ respect to weak$^*$ topology.
\begin{lemma}\label{Lemm_cono_denso_en_cono_bidual} 
Let $C$ be a cone in a normed space $X$. Then $\overline{C}^{\mbox{weak}^*}=C^{**}$.
\end{lemma}
\begin{proof}
Since $C\subset C^{**}$ and $C^{**}$ is weak$^*$-closed, it follows that $\overline{C}^{\mbox{weak}^*}\subset C^{**}$. We will prove the other inclusion by contradiction. Assume there exists $T \in C^{**} \setminus \overline{C}^{\mbox{weak}^*}$. By the Hahn-Banach theorem, there exist $f\in X^*$ and $\alpha \in \R$ such that 
\begin{equation}\label{eq_lem}
T(f)<\alpha \leq \inf\{f(x)\colon x \in  \overline{C}^{\mbox{weak}^*}\}\leq  \inf\{f(c)\colon c \in  C\}\leq 0.
\end{equation}
If $\inf\{f(c)\colon c \in  C\}=0$, then $f\in C^*$ which contradicts $T \in C^{**}$ because by (\ref{eq_lem}) we have $T(f)<\alpha \leq 0$. Consequently $\inf\{f(c)\colon c \in  C\}<0$. Hence, there exist $r<0$ and $c\in C$ such that $f(c)=r<0$. Since $nc\in C$, $\forall n \geq 1$, it follows that $\alpha \leq \inf\{f(c)\colon c \in  C\}\leq f(nc)=nr<0$, $\forall n \geq 1$, which is impossible.
\end{proof}
The following result connects extremity properties of the bidual cone to the existence of small slices containing the origin in the primal cone. This is the key to link cones with a large dual $C^*$, to  dentability properties of $C$ --as we will see in the proof of Theorem \ref{Teorema_caracte_denting}--.
\begin{proposition}\label{Prop_bidual_pointed_implica_slices_base}
Let $X$ be a normed space and $C\subset X$ a cone. If $C^{**}$ is pointed, then for every $R>0$ and $C_R:=C\cap \overline{B_R(\cero)}$, the family of open slices containing $\cero$ forms a neighbourhood base for $\cero$ relative to $(C_R,\mbox{weak})$.
\end{proposition}
\begin{proof}
Fix $R>0$ and $C_R \subset C$ as in the statement. Let us denote  by $\parallel \cdot \parallel_{**}$ the norm on the bidual $X^{**}$ and define the weak$^*$-compact set
\[C_R^{**}:=\{T\in C^{**}\colon \parallel T\parallel_{**} \leq R\}\subset X^{**}.\]
Clearly $C_R=C_R^{**}\cap X$ and $\cero$ is an extreme point of $C_R^{**}$. Let $W$ be a weak-neigbourhood of $\cero$ in $X$. It is not restrictive to assume that there exist $\{f_i\}_{i=1}^n\subset X^*$ and $\lambda\in \R$ such that $W=\cap_{i=1}^n\{x \in X \colon f_i(x)<\lambda\}$. We define the weak$^*$ open set $W^{**}:=\cap_{i=1}^n\{T\in X^{**} \colon T(f_i)<\lambda\}.$ 
By Choquet's lemma there exist $f \in X^*$ and $\alpha \in \R$ such that $\cero \in \{T\in C_R^{**}\colon T(f)< \alpha\} \subset  W^{**} \cap C_R^{**}$. Clearly, $\cero \in \{x \in C_R \colon f(x)<\alpha\} \subset W\cap C_R$, which finishes the proof. 
\end{proof}

\begin{proof}[Proof of Theorem \ref{Teorema_caracte_denting}] 

(i)$\Rightarrow$(ii).  By (i)$\Rightarrow$(ii) in Theorem \ref{Tma_JMAA}, $C$ has a bounded slice $S$. By \cite[Lemmas~2.1 and 2.2]{GARCIACASTANO20151178}, there exist $f\in X^*$ and $\lambda>0$ such that $S=\{f<\lambda\}\cap C$. Finally, by \cite[Proposition 2.4]{GARCIACASTANO20151178}, $f\in C^*$.

(ii)$\Rightarrow$(i). Assume (ii) and fix $\lambda:=\min\{\lambda_i\colon 1\leq i \leq n\}>0$. Now, we define the relatively weakly open set $V:=\bigcap_{i=1}^n\{f_i<n\lambda \}\cap C$ in $C$ and the functional $f:=\frac{1}{n}\sum_{i=1}^nf_i$. It is clear that $V$ is bounded, $f\in C^*$, and $\{f<\lambda\}\cap C \subset V$. Then, $C$ has a bounded slice. Finally, (ii)$\Rightarrow$(i) in Theorem \ref{Tma_JMAA} applies. 

(i)$\Leftrightarrow$(iv). This is a consequence of (i)$\Leftrightarrow$(iii) in Theorem \ref{Tma_JMAA} and Lemma \ref{lema_order_unit} (i).

(i)$\Rightarrow$(iii). Clearly $\cero$ is a point of continuity for $C$. On the other hand, the former equivalence, (i)$\Leftrightarrow$(iv), implies that $C^*$ has an order unit. Now, by Lemma~\ref{lema_order_unit}~(i), we can assume that Int $C^*\not = \emptyset$. Therefore, $C^*-C^*$ is a subspace of $X^*$ with a nonempty interior. Hence, $C^*-C^*=X^*$ which yields $C\in \mathrm{q}\mathcal{G}^*$. 

(iii)$\Rightarrow$(i). Consider again $\overline{C}^{\mbox{weak}^*}$, i.e.  the closure of $C$ in $X^{**}$ respect to the  weak$^*$ topology.  Assume $C\in $ q$\mathcal{G}^*$ and pick any $x^{**}\in \overline{C}^{\mbox{weak}^*}\cap (-\overline{C}^{\mbox{weak}^*})\subset X^{**}$. Then we have, simultaneously $x^{**}(f)\geq 0$ and $x^{**}(-f)\geq 0$, $\forall f \in C^*$. Hence $x^{**}=0$, which means that the cone $\overline{C}^{\mbox{weak}^*}$ in $X^{**}$ is pointed. Then, by Lemma~\ref{Lemm_cono_denso_en_cono_bidual}, $C^{**}$ is pointed. Let us fix now $\epsilon>0$. Next, we will find a slice on $C$ containing $\cero$ and having diameter smaller than $\epsilon$. By $W$ we denote a weak neighbourhood of $\cero$. Now, since $\cero$ is a point of continuity for $C$, we assume that $\cero \in W\cap C \subset C_{\frac{\epsilon}{4}}$ (making use of the notation in Proposition \ref{Prop_bidual_pointed_implica_slices_base}). Now, applying Proposition \ref{Prop_bidual_pointed_implica_slices_base} to $W$ and $R=\frac{\epsilon}{2}$, we have $f\in X^*$ and $\lambda>0$ such that $\{f<\lambda\}\cap C_{\frac{\epsilon}{2}} \subset W \cap C_{\frac{\epsilon}{2}}\subset C_{\frac{\epsilon}{4}}$. Hence $\{f<\lambda\}\cap C \subset C_{\frac{\epsilon}{2}}$. Indeed, if $c\in \{f<\lambda\}\cap C$ and $\parallel c \parallel > \frac{\epsilon}{2}$, then $c_0=\frac{3\epsilon}{8}\frac{c}{\parallel c \parallel} \in \{f<\lambda\}\cap C_{\frac{\epsilon}{2}} \setminus C_{\frac{\epsilon}{4}}$, a contradiction.

(iv)$\Rightarrow$(v). There is nothing to prove.

(v)$\Rightarrow$(i). By Lemma \ref{lema_order_unit} (ii), $\mathrm{Int}C^*\not = \emptyset$. Now, (iii)$\Rightarrow$(i) in Theorem \ref{Tma_JMAA} applies.
\end{proof}

Let us recall the following notion from \cite{Guirao2014}, notion which was used --at least implicitly-- by Kunen and Rosenthal in \cite{Kunen1982}. Let $A$ be a subset of a normed space $X$, a point $x \in A$ is called a \emph{weakly strongly extreme point} of $A$ if given two sequences $(a_n)_n$ and $(a'_n)_n$ in $A$ such that $\lim_n(a_n+a'_n)=2x$, then weak-$\lim_n a_n=x$. Now, we provide the local version of the former notion. We say that $x \in A$ is a \emph{locally weakly strongly extreme point} of $A$ if there exists a neighbourhood $U$ of $x$ such that  given two sequences $(a_n)_n$ and $(a'_n)_n$ in $A\cap U$ such that $\lim_n(a_n+a'_n)=2x$, then weak-$\lim_n a_n=x$. Such a notion will allow us to solve \cite[Problem~2.6]{GARCIACASTANO20151178} as we will see after the next result.
\begin{theorem}\label{Tma_Prop_Vicente_para_Cono_mejorada}
Let $X$ be a normed space, $C\subset X$ a pointed cone, $R>0$, and $C_R:=C\cap \overline{B_R(\cero)}$. Assume that the family of open slices containing $\cero$ forms a neighbourhood base for $\cero$ relative to $(C_R,\mbox{weak})$. Then $\cero$ is a locally weakly strongly extreme point of $C$.
\end{theorem}
\begin{proof}
Let us consider two sequences $(c_n)_n$ and $(c'_n)_n$ in $C_R$ such that $\lim_{n \rightarrow +\infty}(c_n+c'_n)=\cero$. We will show that weak-$\lim_{n \rightarrow +\infty}c_n=\cero$. We claim that for every subsequence $(c_{n_k})_{k}$ of $(c_n)_n$, we have
\begin{equation}\label{ec_prop}
\cero \in \overline{\{c_{n_k}\colon k \geq 1\}}_,^{(C_R,\mbox{\tiny weak})}
\end{equation}
claim which implies weak-$\lim_{n \rightarrow +\infty}c_n=\cero$. Indeed, if the former limit either does not exist or it is not $\cero$, then there would exist $f\in X^*$ and $\lambda>0$ such that for every $k \geq 1$ we could choose $n_k \geq k$ such that the corresponding $c_{n_k}$ does not belong to the set $\{f<\lambda\}\cap C_R$. Then we could construct a subsequence $(c_{n_k})_{k}$ of $(c_n)_n$ verifying $c_{n_k} \not \in \{f<\lambda\}\cap C_R$ for every $k \geq 1$, contradicting (\ref{ec_prop}). Therefore we only need to show the above claim. For that purpose it suffices to prove (\ref{ec_prop}) for the initial sequence $(c_n)_n$ (the argument for any subsequence is the same). We will show it by contradiction. Assume that (\ref{ec_prop}) is not true for $(c_n)_n$. Then, there exist $f\in X^*$ and $\lambda>0$ such that 
\begin{equation}\label{ec2_prop_vicente_para_cono_mejorada}
c_{n} \not \in \{f<\lambda\}\cap C_R,\, \forall n \geq 1.
\end{equation}
However, by hypothesis there exists $n_0 \geq 1$ such that $c_n+c'_n \in \{f<\lambda\}\cap C_R$, $\forall n \geq n_0$; which by convexity yields $c'_n \in \{f<\lambda\}\cap C_R$, $\forall n \geq n_0$. Hence weak-$\lim_{n \rightarrow +\infty}c'_n=\cero$. Indeed, if we consider any $g \in X^*$ and $\mu>0$ there exists $n_1$ such that $c_n+c'_n \in \{g<\mu\}\cap \{f<\lambda\}\cap C_R$, $\forall n \geq n_1$. Thus \[c'_n \in \{g<\mu\}\cap \{f<\lambda\}\cap C_R\subset \{g<\mu\}\cap C_R,\,\forall n \geq n_1.\] Finally, weak-$\lim_{n \rightarrow +\infty}c'_n=\cero$ implies weak-$\lim_{n \rightarrow +\infty}c_n=\cero$ contradicting (\ref{ec2_prop_vicente_para_cono_mejorada}). Contradiction which comes from assuming that (\ref{ec_prop}) is not true for $(c_n)_n$. 
\end{proof}
\begin{remark}\label{Remark_Problem_Prop_Vicente_Mejorada}
Theorem \ref{Tma_Prop_Vicente_para_Cono_mejorada} solves \cite[Problem 2.6]{GARCIACASTANO20151178} in the negative because we reach the equivalence between the corresponding statements in \cite[Proposition 2.5]{GARCIACASTANO20151178} when we change condition (i) in \cite[Proposition 2.5]{GARCIACASTANO20151178} by ``$\cero$ is a locally weakly strongly extreme point of $C$''  even though $C$ is not closed.
\end{remark}
Let us remind from the introduction that \cite[Theorem]{Lin1988} was used in \cite{Daniilidis2000} to answer Gong's question in the context of Banach spaces. In the last part of this section we state and prove a new version of \cite[Theorem]{Lin1988} --labelled as Theorem~\ref{Teorema_caracte_denting_point}-- in which we have dropped down the hypothesis of completeness. In particular, we establish a characterization of dentability in terms of the notion of point of continuity in normed spaces and for arbitrary convex sets.  Let us fix a convex set $A$ and $x\in A$. Let us recall that $x$ is an extreme point of $A$ if $x$ does not belong to any non degenerate line segment in $A$. Fixed a subset $A \subset X$, we will denote by $\widetilde{A}$ the closure of $A$ in $X^{**}$ respect to the weak$^*$ topology. An extreme point $x$ of $A$ is said to be a \emph{preserved extreme point} of $A$ if $x$ is also an extreme point of the set $\widetilde{A}$ (see \cite{Guirao2014}).  The next result clarifies the relationship between the notion of preserved extreme point and that of locally weakly strongly extreme point. Such a result is a reformulation of \cite[Proposition 2.2]{Guirao2014} for a convex set (instead of the unit ball) in a normed space and it will be used in the proof of Theorem \ref{Teorema_caracte_denting_point}. From now on, $A_R(x):=A\cap \overline{B_R(x)}$, $\forall R>0$ and $x \in A$, being $A \subset X$ a subset.
\begin{proposition}\label{Prop_Vicente_mejorada}
Let $X$ be a normed space, $A\subset X$ a convex subset, and $x \in A$. Consider the following properties.
\item[(i)] $x$ is a locally weakly strongly extreme point of $A$.
\item[(ii)] $x$ is a preserved extreme point of $A$.
\item[(iii)] The family of open slices containing $x$ forms a neighbourhood base for $x$ relative to $(A_R(x),\mbox{weak})$, $\forall R>0$.
\bigskip

Then we have (i)$\Rightarrow$(ii)$\Rightarrow$(iii). Moreover, if $A$ is also closed, then the three properties above are equivalent.
\end{proposition}
\begin{proof}
The proof of implication (iii)$\Rightarrow$(i) (resp. of (i)$\Rightarrow$(ii), resp. of (ii)$\Rightarrow$(iii)) in \cite[Proposition~9.1]{Guirao2014}  applied to each $A_R(x)$ proves our implication (i)$\Rightarrow$(ii) (resp. our implication (ii)$\Rightarrow$(iii), resp. our implication (iii)$\Rightarrow$(i) under the extra assumption of $A$ being closed).
\end{proof}
\begin{theorem}\label{Teorema_caracte_denting_point}
Let $X$ be a normed space, $A\subset X$ a convex subset, and $x\in A$. The following are equivalent:
\item[(i)] $x$ is a denting point of $A$.
\item[(ii)] $x$ is a point of continuity and a weakly strongly extreme point of $A$.
\item[(iii)] $x$ is a point of continuity and a preserved extreme point of $A$.
\end{theorem}
The equivalence between (i) and (iii) above was already proved in \cite{Lin1989} but for bounded, closed and convex sets. Another generalization of \cite[Theorem]{Lin1988} can be found in \cite{Ben}, but stated in the setting of Banach spaces.
\begin{proof}
(i)$\Rightarrow$(ii) is consequence of the fact that every denting point is a weakly strongly extreme point. 

(ii)$\Rightarrow$(iii) If $x \in A$ is weakly strongly extreme point of $A$, then so it is of each set $A_R(x):=A\cap \overline{B_R(x)}$ and $R>0$. Now, applying the former proposition, we conclude that $x$ is a preserved extreme point of $A$.

(iii)$\Rightarrow$(i) Fix $\epsilon>0$ and a weak open set $W\subset X$ which contains $x$ such that diam$(W\cap A)<\epsilon/2$. By assertion (iii) of Proposition \ref{Prop_Vicente_mejorada}, there exists $f\in X^*$ and $\lambda \in \R$ such that $H:=\{f<\lambda\}$ verifies $x\in H\cap A_{\epsilon}(x) \subset W\cap A_{\epsilon}(x)\subset A_{\epsilon/2}(x)$. We claim that $H\cap A \subset A_{\epsilon}(x)$. Otherwise, we could choose $y \in A$ such that  $f(y)<\lambda$ and $\parallel y-x \parallel >\epsilon$. For every $\alpha \in [0,1]$, we define $y_{\alpha}:=\alpha y + (1-\alpha)x \in A$. By convexity $f(y_{\alpha})<\lambda$, $\forall \alpha \in [0,1]$. Moreover, there exist $\alpha_0,\alpha_1 \in (0,1)$ such that $\epsilon/2<\parallel y_{\alpha}-x\parallel < \epsilon$, for every $\alpha \in (\alpha_0,\alpha_1)$. Fix some $\alpha^*\in (\alpha_0,\alpha_1)$, then $y_{\alpha^*} \in H \cap A_{\epsilon}(x)$ but $y_{\alpha^*} \not \in A_{\epsilon/2}(x)$, a contradiction. Thus, diam$(H\cap A)<2\epsilon$ and the proof is over.
\end{proof} 
\section{Some open problems solved for cones in $\mathrm{q}\mathcal{G}^*$}\label{Secc_Normal_cones}
In this section we will see how several open problems in the literature regarding cones can be answered using (i)$\Leftrightarrow$(iii) in Theorem \ref{Teorema_caracte_denting} when the corresponding cones belong to $\mathrm{q}\mathcal{G}^*$. Besides, in Section~\ref{Secc_Examples}, we will check that $\mathrm{q}\mathcal{G}^*$ contains the canonical order cones of some well known noncomplete normed spaces. Next, we will introduce the notion of normal cone which provides a wide subfamily of cones in $\mathrm{q}\mathcal{G}^*$. Given a cone $C\subset X$, we say that a subset $A\subset X$ is \emph{full} (respect to the order given by $C$) if for each pair $x$, $y\in A$ we have $[x,y]\subset A$; the cone $C$ is said to be \emph{normal} if there exists a neighborhood base (for the topology given by the norm on $X$) at $\cero$ consisting of full sets --for more information about normal cones we refer the reader to \cite{ali-tour}--. It is not difficult to prove that dentability of a cone at the origin implies normality of such a cone. On the other hand, by \cite[Theorem~4.4]{Abramovich1992} any normal cone belongs to $\mathrm{q}\mathcal{G}^*$ and by \cite[Lemma 2.39]{ali-tour} any pointed cone in a normed vector lattice is normal. 
\begin{remark}\label{Rem_quasi_gen_dual_contiene_normal}
The family $\mathrm{q}\mathcal{G}^*$ contains the family of normal cones and so any arbitrary pointed cone in a normed vector lattice.
\end{remark}
It is worth pointing out that vector lattices provide a useful unifying framework for many problems on convex programming (see \cite{Borwein1992}).

Now, we will pay attention to Gong's question in \cite[Conclusions]{Gong1995}. Gong asked if for any normed space $X$ and any closed cone $C$ the condition $\cero \not \in  \overline{C\setminus \epsilon B_X}^{\mbox{\small weak}}$, $\forall \epsilon \in (0,1)$ is really weaker than the cone $C$ has a bounded base. Let us recall that a \emph{base} for a cone $C$ is a nonempty convex subset $B\subset C$ such that $0_X \not \in \overline{B}$ and each $c\in C\setminus \{\cero\}$ has a unique representation of the form $c=\lambda b$ for some $\lambda>0$ and $b\in B$. A base is called a \emph{bounded base} if it is a bounded subset of $X$. Bounded bases have been widely studied and provide a useful tool in many topics such as in the theory of Pareto efficient points in \cite{Borwein1993}, in reflexivity of Banach spaces in \cite{Casini2010a}, and in density theorems in \cite{Gong1995}. It is known that a cone $C$ has a bounded base if and only if the origin is a denting point of $C$. In \cite[Example~1.5]{GARCIACASTANO20151178} we provided a non closed cone for which the answer to Gong's question is positive. In this paper, we provide a wide family of cones for which the answer to Gong's question is negative.  Applying (iii)$\Rightarrow$(i) of Theorem \ref{Teorema_caracte_denting}, we can state.
\begin{corollary}\label{Coro_Gong}
Let $X$ be a normed space and $C\subset X$ a cone which belongs to $\mathrm{q}\mathcal{G}^*$. Then, the condition $\cero \not \in  \overline{C\setminus \epsilon B_X}^{\mbox{\small weak}}$, $\forall \epsilon \in (0,1)$, is equivalent to the condition that the cone $C$ has a bounded base.
\end{corollary}

On the other hand, \cite[Example 4.6]{Abramovich1992} shows that the closedness of a cone $C$ does not imply $C \in \mathrm{q}\mathcal{G}^*$ --even when $X$ is a Banach space--. We state the following problem for a future research.
\begin{problem}
Is $\mathrm{q}\mathcal{G}^*$ maximal among those families of cones for which Corollary \ref{Coro_Gong} holds true?
\end{problem}

Gong's question is directly related to the following two results on density of Arrow, Barankin and Blackwell's type: \cite[Corollary~4.2]{Petschke1990} --due to Petschke-- and \cite[Theorem~3.2~(a)]{Gong1995} --due to Gong--. Both results concern the approximation of the Pareto efficient points of compact convex subsets by points that are maximizers of some strictly positive functional on this set. Given a subset $A\subset X$, the set of \emph{maximal (or efficient) points of $A$} (with respect to the cone $C$) is defined as $Max(A,C):=\{x \in A \colon \{x\}=A\cap (x+C)\}$. On the other hand, given a cone $C$ the \emph{set of all strictly positive functionals} is defined by $C^{\#}:=\{f \in X^*\colon f(c)>0,\, \forall c \in C, \, c\not = \cero\}$, and the set of \emph{positive (or proper efficient) points} of $A$ as $Pos(A,C):=\{x\in A \colon \exists f \in C^{\#},\, f(x)=\sup_A f\}$. It is straightforward that $Pos(A,C)\subset Max(A,C)$, however this inclusion is strict. Next, we state the density results of Petschke and Gong.
\begin{theorem}(Petschke)\label{Tma_Petschke}
Let $A$ be a weak compact convex subset of the normed space $X$ and assume that $C$ is a closed cone with a bounded base. Then $Max(A,C)\subset \overline{Pos(A,C)}^{\parallel \cdot \parallel}$.
\end{theorem}
\begin{theorem}(Gong)\label{Tma_Gong}
Let $A$ be a weak compact convex subset of the normed space $X$ and assume that  $C$ is a closed cone with a base such that $\cero$ is a point of continuity for $C$. Then $Max(A,C)\subset \overline{Pos(A,C)}^{\parallel \cdot \parallel}$.
\end{theorem}

Daniilidis stated in \cite[Corollary~2]{Daniilidis2000} that Theorems \ref{Tma_Petschke} and \ref{Tma_Gong} are equivalent when $X$ is a Banach space. However, the question about the equivalence of such theorems still remains open in case the norm on $X$ is not complete. Applying Corollary \ref{Coro_Gong}, we can state.

\begin{corollary}\label{Coro_density_theorems}
Theorems \ref{Tma_Petschke} and \ref{Tma_Gong} are equivalent for cones in $\mathrm{q}\mathcal{G}^*$.
\end{corollary}

In \cite{Kountzakis2006}, the authors proved the equivalence of dentability and point of continuity at the origin of a cone $C$ in a normed space  under the extra assumption that qi$C^*\not = \emptyset$. They also stated there Problem 5 in which they asked: If $\cero$ is a point of continuity of a cone $C\subset X$ then qi$C^*\not = \emptyset$? By Theorem \ref{Teorema_caracte_denting}, if $C\in \mathrm{q}\mathcal{G}^*$ and $\cero$ is a point of continuity of $C$, then $\cero$ is a denting point and, as a consequence, qi$C^*\not = \emptyset$. Therefore, we have a positive answer of Problem 5 in \cite{Kountzakis2006} for cones in $\mathrm{q}\mathcal{G}^*$. Next, we state it as a corollary.
\begin{corollary}\label{Coro_Kountz}
Let $X$ be a normed space and $C\subset X$ a cone which belongs to $\mathrm{q}\mathcal{G}^*$. If $\cero$ is a point of continuity for $C$, then qi$C^*\not = \emptyset$.
\end{corollary}

We finish this section showing some applications of dentability to operators theory. A \emph{Krein space} is an ordered Banach space whose order cone is closed and has order units. Krein spaces have been widely studied in connection with operators theory, spectral theory, and fixed point theory --we refer the reader to \cite{Abramovich1992} for a survey regarding positive operators on Krein spaces--. Given a normed space $X$ and $C \subset X$ a quasi-generating cone, it is easy to prove that the dual cone $C^*$ is pointed. In addition, if the origin is denting in $C$, then --by (i)$\Rightarrow$(iv) of Theorem \ref{Teorema_caracte_denting}-- the dual cone $C^*$ has an unity. Therefore, under the former assumptions the dual space $X^*$ with the order given by $C^*$ is a Krein space.  In the next result, statement (i) is a consequence of \cite[Theorem 2.32]{ali-tour} and \cite[Theorem~6.4]{Abramovich1992}. Statement (ii) is a consequence of \cite[Corollary 7.6]{Abramovich1992}.
\begin{corollary}\label{Coro_Aplic_Teor_Oper}
Let $X$ be a normed space ordered by a quasi-generating cone $C$. If $\cero$ is a denting point of $C$, then the following statements hold true:
\item[(i)] Every linear and positive operator $T:X^* \rightarrow X^*$ is continuous. In addition, if $T$
is not a multiple of the identity, then it has a nontrivial hyperinvariant subspace.
\item[(ii)] If a positive contraction $T:X^* \rightarrow X^*$ has $1$ as an eigenvalue, then there exits   a fixed point $0<f\in X^{**}$ for the adjoint operator of $T$. 
\end{corollary}

\section{The geometry of some canonical order cones}\label{Secc_Examples}
In this section we will consider the canonical order cones of some  well known noncomplete normed spaces. We will verify that each one belongs to q$\mathcal{G}^*$ and --applying (i)$\Leftrightarrow$(iii) in Theorem~\ref{Teorema_caracte_denting}-- we will conclude that the origin is not a point of continuity. Next, we introduce a result in order to avoid working in dual spaces when checking the condition $C \in \mathrm{q}\mathcal{G}^*$
. 
\begin{proposition}\label{Prop_weakly_imply_dual_quasi_generating}
Let $X$ be a normed space and $C\subset X$ a pointed cone. If $\cero$ is a locally weakly strongly extreme point of $C$, then $C \in \mathrm{q}\mathcal{G}^*$.
\end{proposition}

\begin{proof}
The proof of the implication (i)$\Rightarrow$(ii) in \cite[Proposition~2.5]{GARCIACASTANO20151178} can be easily adapted to prove that: if $\cero$ is a locally weakly strongly extreme of $C$, then $\overline{C}^{\mbox{weak}^*}(=C^{**}$ by Lemma \ref{Lemm_cono_denso_en_cono_bidual}) is pointed in $X^{**}$. Thus, since $C^*-C^*$ is a subspace of $X^*$, if $x^{**} \in X^{**}$ satisfies $x^{**}|_{(C^*-C^*)}\equiv 0$ then $x^{**}|_{C^*}\equiv 0$, so $x^{**}\in C^{**}\cap (-C^{**})=\{\cero\}$. This proves that $C^*-C^*$ is dense in $X^*$.
\end{proof}

Next, we provide the three examples announced at the beginning of this section. Let us state the first one.
\begin{example}\label{example_C00}
Let $\Gamma$ be a nonempty set, consider the vector space $c_{00}(\Gamma)$ of those $(x_{\gamma})_{\gamma \in \Gamma} \in l_{\infty}(\Gamma)$ such that $\{\gamma \in \Gamma \colon x_{\gamma}\not =0\}$ is finite, the noncomplete normed space $(c_{00}(\Gamma),\norm_{\infty})$ where $\parallel (x_{\gamma})_{\gamma \in \Gamma}\parallel_{\infty}:=\sup\{|x_{\gamma}|\colon \gamma \in \Gamma\}$, and the order cone $c_{00}(\Gamma)^+:=\{(x_{\gamma})_{\gamma \in \Gamma} \in c_{00}(\Gamma) \colon x_{\gamma} \geq 0,\, \forall\gamma \in \Gamma\}.$ Then $c_{00}(\Gamma)^+\in \mathrm{q}\mathcal{G}^*$
and the origin is not a point of continuity for $c_{00}(\Gamma)^+$.
\end{example}
\begin{proof}
We first prove that the dual cone $(c_{00}(\Gamma)^+)^* \subset  (c_{00}(\Gamma),\norm_{\infty})^*$ is quasi-generating. We will make use of Proposition \ref{Prop_weakly_imply_dual_quasi_generating}.  Consider two sequences $(c_n)_n$, $(d_n)_n$ in $c_{00}(\Gamma)^+$ such that $\displaystyle \lim_{n \rightarrow +\infty}\parallel c_n+ d_n\parallel_{\infty}=0$. We claim that $\displaystyle \lim_{n \rightarrow +\infty}\parallel c_n\parallel_{\infty}=0$. Indeed, for every $n \in \N$, we write $c_n=(c_{\gamma}^n)_{\gamma \in \Gamma}$ and $d_n=(d_{\gamma}^n)_{\gamma \in \Gamma}$. Clearly $0 \leq c_{\gamma}^n\leq c_{\gamma}^n+d_{\gamma}^n$, for every $n\geq 1$ and $\gamma\in \Gamma$. Thus, $0 \leq \parallel c_n\parallel_{\infty} \leq \parallel c_n+d_n\parallel_{\infty}$, $\forall n \geq 1$, which yields $\displaystyle \lim_{n \rightarrow +\infty}\parallel c_n\parallel_{\infty}=0$. Hence, weak-$\lim_{n \rightarrow +\infty}c_n=0$.

In order to prove that the origin is not a point of continuity for $c_{00}(\Gamma)^+$ we will use the equivalence (i)$\Leftrightarrow$(iii) of Theorem \ref{Teorema_caracte_denting}. Hence, it is sufficient to show that the origin is not a denting point of $c_{00}(\Gamma)^+$. For that purpose we will check that Int$(c_{00}(\Gamma)^+)^*=\emptyset$ and, applying the equivalence (i)$\Leftrightarrow$(iii) in Theorem \ref{Tma_JMAA}, the proof will be completed. Now, the density of $c_{00}(\Gamma)$ in $c_{0}(\Gamma)$  provides the equality 
$(c_{00}(\Gamma)^+)^*=\ell_{1}(\Gamma)^+$ --see \cite{Fabian2001}--, where $\parallel (x_{\gamma})_{\gamma \in \Gamma}\parallel_1 = \sum_{\gamma \in \Gamma}|x_{\gamma}|=\sup\{\sum_{\gamma \in F}|x_{\gamma}|\colon F \subset \Gamma \mbox{ is finite}\}$. Now we claim that $\mathrm{Int}(\ell_{1}(\Gamma)^+)=\emptyset$. Indeed, let us fix $x=(x_{\gamma})_{\gamma \in \Gamma}\in \ell_{1}(\Gamma)^+$, $\epsilon>0$, and $\gamma_0 \in \Gamma$ such that $0\leq x_{\gamma_0}< \epsilon/2$. Define $y=(y_{\gamma})_{\gamma \in \Gamma} \in \ell_{1}(\Gamma)$ as $y_{\gamma}:=x_{\gamma}$ for $\gamma\not =\gamma_0$ and $y_{\gamma_0}:=-\epsilon/2$. Then $y \not \in \ell_{1}(\Gamma)^+$ and $\parallel x-y\parallel_1 <\epsilon$.
\end{proof}

Next, we will study spaces of differentiable functions defined on some bounded real interval. Let us recall that for every continuous real function $f$ on a topological space $X$, it is defined the \emph{support of $f$} --written supp($f$)-- as the closure of the set $\{x \in X\colon f(x)\not =0\}$. 
\begin{example}\label{example_C^k}
Let us fix any $k \geq 1$, consider the vector space $C^{k}[a,b]$ of all functions on $[a,b]$ that have $k$ continuous derivatives, the noncomplete normed space $(C^k[a,b],\normsup)$ where $\parallel f \parallel_{\infty}:=\sup\{|f(x)|\colon x \in [a,b]\}$, and the order cone $C^k[a,b]^+:=\{f \in C^k[a,b] \colon f(x)\geq 0,\, \forall x \in [a,b]\}$. Then $C^k[a,b]^+ \in \mathrm{q}\mathcal{G}^*$ and the origin is not a point of continuity for $C^k[a,b]^+$.
\end{example}
\begin{proof}
Let us prove only the case $k=1$. The relation $C^k[a,b]^+ \in \mathrm{q}\mathcal{G}^*$ can be proved using again Proposition~\ref{Prop_weakly_imply_dual_quasi_generating} with an argument similar to that we used in the first paragraph of the proof in the previous example. 

In order to prove the last claim in the statement we will use again the equivalence (i)$\Leftrightarrow$(iii) of Theorem \ref{Teorema_caracte_denting}. Thus, it suffices to prove that the origin is not a denting point of $C^1[a,b]^+$. The equivalence (i)$\Leftrightarrow$(vi) of \cite[Theorem 1.1]{GARCIACASTANO20151178} and the statement (i) of \cite[Corollary 1.2]{GARCIACASTANO20151178} implies that if the origin is a denting point of $C^1[a,b]^+$, then every sequence in $C^1[a,b]^+$ converging to the origin in the weak topology also converges in norm. We will show that this is not the case. For that purpose, we consider a decreasing sequence $\{a_n\}_n$ in $(a,b]$ that converges to $a$. Define a bounded sequence $\{f_n\}_n$ in $C^1[a,b]$ of positive functions vanishing at $a$, such that supp($f_n$) $\subset [a_{n+1},a_{n-1}]$, and $f_n(a_n)=1$, $\forall n \in \N$. Then the sequence $\{f_n\}_n$ converges to the origin pointwise on $[a,b]$, and it is bounded, so it converges to the origin in $C[a,b]$ in the weak topology (see \cite[Example 2, p. 112]{Reed1980}). Now the density of $C^1[a,b]$ in $C[a,b]$ and the Hahn-Banach theorem yields that $\{f_n\}_n$ converges to the origin in $C^1[a,b]$ in the weak topology, although not in norm. It follows that the origin is not a denting point of  $C^1[a,b]^+$.
\end{proof}

Next, we will examine the behaviour of the space of real functions with compact support.  
\begin{example}\label{example_C00(R)}
Let us consider the vector space $C_{00}(\R)$  of those real continuous functions on $\R$ for which supp($f$) is compact, the noncomplete normed space $(C_{00}(\R),\normsup)$ where $\parallel f \parallel_{\infty}:=\sup\{|f(x)|\colon x \in \R\}$, and the order cone $C_{00}(\R)^+:=\{f \in C_{00}(\R) \colon f(x)\geq 0,\, \forall x \in \R\}$. Then $C_{00}(\R)^+\in \mathrm{q}\mathcal{G}^*$ and the origin is not a point of continuity for $C_{00}(\R)^+$.
\end{example}
\begin{proof}
Again by Proposition \ref{Prop_weakly_imply_dual_quasi_generating}, the cone $(C_{00}(\R)^+)^*$ is quasi-generating. 

As in the proof of the previous examples, it suffices to prove that the origin is not a denting point of $C_{00}(\R)^+$. The set $\{\pm \delta_x\colon x \in \R\}$ is a James boundary in $C_{00}(\R)^*$, in other words, given $f \in C_{00}(\R)$ we have $\parallel f \parallel_{\infty}=|f(x_0)|=|\delta_{x_0}(f)|$, for some $x_0 \in \R$. Then the convergence of any sequence $(f_n)_n$ in $C_{00}(\R)$ in the weak topology is given by the convergence in $\R$ of  $\delta_x(f_n)$ for every $x \in \R$, or equivalently, it is given by the pointwise convergence of  $(f_n)_n$ on $\R$ (see \cite{Fabian2001}). Again, we can define a bounded sequence $(f_n)_n$ in $C_{00}(\R)^+$ that converges to the origin pointwise but not in norm. Hence, such a sequence converges to the origin in the weak topology and, as a consequence, the origin is not a denting point of $C_{00}(\R)^+$. 
\end{proof}

\begin{acknowledgements}
We thank the referees for their suggestions which have helped us to improve the overall aspect of the manuscript.
\end{acknowledgements}

\end{document}